\documentclass[11pt]{article}

\usepackage[T1]{fontenc}
\usepackage{lmodern}
\usepackage{microtype}
\usepackage{amsmath,amssymb,amsthm,mathtools}
\usepackage{needspace, todonotes, comment}
\usepackage[hidelinks]{hyperref}
\newtheorem{theorem}{Theorem}[section]
\newtheorem{proposition}[theorem]{Proposition}
\newtheorem{definition}[theorem]{Definition}
\newtheorem{corollary}[theorem]{Corollary}
\newtheorem{lemma}[theorem]{Lemma}
\newtheorem{conjecture}[theorem]{Conjecture}

\newcommand{\F}{\mathbb F}
\newcommand{\dist}{d_{\mathrm H}}

\title{Triangle-Free Sets in Exact-Distance Hamming Graphs}
\author{Bence Csajb\'ok \thanks{Department of Computer Science, 
	ELTE E\"otv\"os Lor\'and University, 
	H-1117 Budapest, P\'azm\'any P.\ stny.\ 1/C, Hungary. This paper was supported by the János Bolyai Research Scholarship of the Hungarian Academy of Sciences and partially by the National Research, Development and Innovation Fund – grant numbers ADVANCED 153080, EXCELLENCE 151504 and SNN 152582. {\texttt{bence.csajbok@ttk.elte.hu}}} \and Padmini Mukkamala\thanks{BME Budapest Institute of Technology and Economics, Budapest, Hungary. Email: {\texttt{padmini.mvs@gmail.com}}} \and Zolt\'an Lóránt Nagy\thanks{ Department of Computer Science, 
	ELTE  E\"otv\"os Lor\'and University, Budapest, Hungary. The author is supported by  the J\'anos Bolyai Research Grant of the Hungarian Academy of Sciences and  partially by the NRDI EXCELLENCE grant, no.  151504;  	E-mail: {\tt zoltan.lorant.nagy@ttk.elte.hu.}}}
\date{}

\begin{document}
\maketitle
\begin{abstract}
Let $T_q(n,r)$ be the largest size of a vertex set inducing a
triangle-free subgraph of the graph on $[q]^n$ whose edges join
words at Hamming distance $r$. We study this quantity for fixed
alphabet size and fixed distance, with particular emphasis on
the ternary cube. Alphabet restriction transfers the binary
fixed-distance theorem of Balogh, Chen, and Li to give
$T_q(n,2h)=O_h(q^n/n^h)$. We extend the known distance-one formula and distance-two order
to arbitrary integer alphabet sizes, obtaining
$T_q(n,1)=2q^{n-1}$ and $T_q(n,2)=\Theta_q(q^n/n)$,
and establish $T_3(n,4)=\Theta(3^n/n^2)$.
For general fixed distance, we prove
\[
 T_q(n,r)=\Omega_{q,r}\!\left(
 \frac{q^n\sqrt{\log n}}{n^{(r+\lfloor r/2\rfloor)/2}}
 \right)\qquad(q\ge3),
\]
while we show that
$T_q(n,r)=\Omega_{q,r}(q^n/n^{\lfloor(q-1)r/q\rfloor})$ holds
for prime-power $q$ and $r\ge2$ via the extended BCH parity-check constructions.
At distance three, we construct coordinate weights avoiding
bounded-coefficient three-term relations and obtain
$$T_q(n,3)\ge q^n/[n\exp(O_q(\sqrt{\log n}))].$$
Our main ternary construction combines cap partitions, multicolor
triangle Ramsey colorings, and a quadratic map to obtain
\[
 T_3(n,3)=\Omega\!\left(
 \frac{\log\log n}{\log n}\,3^n\right).
\]
A coordinate-line counting argument gives the complementary upper
bound $T_3(n,3)\le(5/27+O(n^{-1}))3^n$, together with an explicit
bound for every fixed $q\ge3$. Whether the ternary density tends
to zero remains open.

\end{abstract}

\section{Introduction and statement of results}

Write $[q]=\{0,1,\ldots,q-1\}$. The $q$-ary Hamming graph
$H_q(n)$ has vertex set $[q]^n$, and two vertices $x,y$  are connected if they are at Hamming distance $d_H(x,y)=1$, i.e., the vectors of the corresponding vertices differ in exactly one coordinate.  Note that $H_q(n)$ is  the Cartesian product of $n$ complete graphs $K_q$. A natural generalisation of the Hamming graph is the \emph{ $r$-distance Hamming graph $H_q(n,r)$}, which has vertex set $[q]^n$, and
\[
 xy\in E(H_q(n,r))
 \quad\Longleftrightarrow\quad
 \dist(x,y)=r.
\]

Our aim is to study the maximum size   of a triangle-free set of vertices in the $r$-distance graph of the Hamming graph. 

\begin{definition}[Maximum triangle-free set of vertices in $H_q(n,r)$]
    \[
 T_q(n,r)=\max\bigl\{|A|:A\subseteq[q]^n
       \text{ and }H_q(n,r)[A]\text{ is triangle-free}\bigr\}.
\]
\end{definition} 

This definition was introduced by Mukkamala and Ravi \cite{mukkamala-ravi}, who were mainly interested in the binary case $q=2$.

  In the terminology of extremal graph theory, it can also be expressed as the
vertex Tur\'an number
\[
       T_q(n,r)=\operatorname{ex}_v(K_3,H_q(n,r)),
\]
namely the largest number of vertices inducing no copy of $K_3$ in the
fixed host graph $H_q(n,r)$.  Vertex Tur\'an problems for the Boolean cube
were developed systematically by Johnson and Talbot
\cite{johnson-talbot}, and related questions for the ternary cube were
considered by Li and Hao \cite{li-hao}.  The forbidden configurations in
those works are embedded subcubes or prescribed small configurations,
rather than equilateral triples in an exact-distance graph, so their
results do not apply directly here.

There is an equivalent hypergraph formulation.  Let
$\mathcal H_{q,n,r}$ be the $3$-uniform hypergraph on $[q]^n$ whose
hyperedges are the triples whose three pairwise Hamming distances are
equal to $r$.  Then
\[
       T_q(n,r)=\alpha(\mathcal H_{q,n,r}).
\]
Thus the problem may also be viewed as an independence problem in a
highly structured $3$-uniform hypergraph, as well as a problem in coding theory.

Exact-distance graphs are closely related to graph powers,
with adjacency prescribed by equality of distances rather
than by an upper bound on distance. Their
structure under Cartesian and other standard products, including exact-
distance graphs of hypercubes, was investigated by Bre\v{s}ar, Gastineau,
Klav\v{z}ar, and Togni \cite{bresar-et-al}.  In the Boolean cube, forbidding
all pairs at one prescribed distance is precisely the classical
one-forbidden-distance code problem.  Its development includes the
subset-sum constructions of Cohen and Z\'emor \cite{cohen-zemor} and the
systematic forbidden-distance bounds of Bassalygo, Cohen, and Z\'emor
\cite{bassalygo-cohen-zemor}.  The problem considered here is a genuinely
higher-order relaxation: prescribed-distance pairs may occur abundantly,
but the graph that they span is required to have clique number at most two.

The binary version arose explicitly from the recursive theta-body method of
Castro-Silva, de Oliveira Filho, Slot, and Vallentin
\cite{castrosilva}.  Mukkamala and Ravi developed direct
combinatorial bounds, proving in particular a uniform estimate
$T_2(n,r)=O(r2^n/(n+1))$ for even $r\le n/2$, as well as exponentially
small density when $r$ is a fixed positive fraction of $n$, with $r/n< \frac 23$ and $r$ is even
\cite{mukkamala-ravi}.  For fixed distance, the decisive result is due to
Balogh,  Chen, and Li \cite{balogh-chen-li}.  For a graph $G$ and an integer $s\ge2$, write
\[
   \alpha_s(G)
   =\max\bigl\{|A|:A\subseteq V(G)
                    \text{ and }G[A]\text{ contains no }K_s\bigr\}.
\]
Thus $\alpha_2(G)=\alpha(G)$ and
$T_q(n,r)=\alpha_3(H_q(n,r))$.
In this notation, Balogh, Chen, and Li~\cite{balogh-chen-li}
proved that, for every fixed $s\ge2$ and $h\ge1$,
\begin{equation}\label{eq:BCL}
   \alpha_s(H_2(n,2h))
   =\Theta_{h,s}\!\left(\frac{2^n}{n^h}\right).
\end{equation}
Alphabet restriction and the colorings developed below yield
nonbinary extensions of this statement; see
Corollary~\ref{cor:fixed-clique-even-distances}.
Their proof brings together the  theory of codes with one forbidden
distance \cite{bassalygo-cohen-zemor} and extremal estimates for
sunflower-free set systems.  Since $H_2(n,r)$ is bipartite when $r$ is odd,
the binary $r$-distance problem is therefore complete:
\[
 T_2(n,r)=
 \begin{cases}
 2^n, & r\text{ odd},\\[2mm]
 \Theta_r(2^n/n^{r/2}), & r\text{ even and fixed}.
 \end{cases}
\]

The passage from two symbols to three yields new phenomena.  A binary
equilateral Hamming triangle can use only coordinates in which exactly one
of its three vertices differs from the other two.  Over an alphabet of size
at least three, a coordinate may instead contain three distinct symbols.
This second local configuration creates triangles at every distance and,
for $q=3$, couples the problem to affine lines, cap sets, and zero-sum
Ramsey colorings.  It is precisely here that the ternary theory separates
from its binary precursor.

To highlight the distinction, suppose that
$x,y,z\in[q]^n$ are pairwise at distance $r$.  Let $a_x,a_y,a_z$ count the
active coordinates at which $x,y,z$, respectively, are the unique exceptional
symbols, and let $d$ count the coordinates at which all three symbols are
different.  Comparing the three pairwise distances gives
\[
 a_x=a_y=a_z=:a,
 \qquad 2a+d=r.
\]
This gives rise to the \textit{possible local types}, indexed by
$0\le a\le\lfloor r/2\rfloor$, with active support of size $r+a$.
For $q=2$ only \textit{the type} $a=r/2$ is available; for $q\ge3$ every type whose
support size satisfies $r+a\le n$ is realizable. This classification
organizes both the general estimates and the constructions below.

Our first result is  a universal upper bound for the $q$-ary case.

\begin{proposition}[Binary-subcube bound]\label{prop:binary-transfer}
For every fixed $h\ge1$ and every $q\ge2$,
\[
       T_q(n,2h)=O_h\!\left(\frac{q^n}{n^h}\right).
\]
\end{proposition}

At distance two, this upper bound is sharp for every fixed alphabet.
At distance one, we have an exact formula. These give extensions of the work of Mukkamala and Ravi
\cite[Lemmas~4 and~5]{mukkamala-ravi}.

\begin{theorem}[Distances one and two]\label{thm:r1r2}
For every fixed $q\ge2$,
\[
 T_q(n,1)=2q^{n-1}
 \qquad\text{and}\qquad
 T_q(n,2)=\Theta_q\!\left(\frac{q^n}{n}\right).
\]
The lower-bound construction for $r=2$ is an independent set in
$H_q(n,2)$.
\end{theorem}

At Combinatorics 2026, A.~Ravi asked whether
\[
 T_q(n,r)=\Theta_{q,r}\!\left(
 \frac{q^n}{n^{\lfloor r/2\rfloor}}\right)
\]
holds for every fixed $q\ge3$ and $r\ge 1$.
Theorem~\ref{thm:r1r2} confirms this at distances one and two,
whereas Theorem~\ref{thm:ternary-r3} disproves it already for
$q=r=3$. The fixed even-distance question remains open in general
and motivates Conjecture~\ref{conj:even}.

At distance three we use two different constructions.
For an arbitrary fixed alphabet, we construct coordinate weights
avoiding all three-term relations with nonzero coefficients in
$\{\pm1,\ldots,\pm(q-1)\}$. Their weighted sum gives a proper
coloring of $H_q(n,3)$ and hence a large independent set.
For the ternary alphabet, a quadratic construction based on
zero-sum Ramsey labelings yields substantially larger
triangle-free sets.

\begin{theorem}[Avoiding distance $3$]\label{thm:checksum}
For every fixed integer $q\ge2$, 
\[
 \alpha(H_q(n,3))
 \ge
 \frac{q^n}{n\exp(O_q(\sqrt{\log n}))}.
\]
In particular,
\[
 T_q(n,3)\ge q^n n^{-1-o(1)}.
\]
\end{theorem} All logarithms without an indicated base are natural.
During the preparation of this manuscript, we became aware of
Yekhanin's recent work~\cite{yekhanin-real-lifts}, which gives
a general construction of large codes for every fixed minimum
Hamming distance at least four. The numerical bound in Theorem~\ref{thm:checksum} can be derived from his distance-four case.
Our construction using relation-avoiding weights was obtained independently, before
the appearance of his preprint. We present it here in the
exact-distance graph setting, where it provides a comparison
with the larger triangle-free sets obtained by our ternary
Ramsey construction.

Theorem~\ref{thm:checksum} deliberately forbids every distance-three edge.
For the ternary alphabet, we can obtain a substantially stronger bound by forbidding triangles
only.  The reason is that the type $a=0$ is an affine line
\[
       x,\quad x+u,\quad x-u,
       \qquad \operatorname{wt}(u)=3,
\]
whereas the type $a=1$ can be separated by a linear statistic.  A quadratic
statistic built from a zero-sum-free labeling of the edges of a complete
graph then eliminates the remaining affine-line triangles.

To state the result, let us recall that a  cap set  in $\F_3^m$ is a vertex set which does not contain $3$ collinear points.
Let $M(m)$ be the maximum size of a cap in $\F_3^m$ and
put
\begin{equation}\label{eq:cap-capacity}
 \kappa_3=\lim_{m\to\infty}M(m)^{1/m},
 \qquad \alpha=\log_3\kappa_3.
\end{equation}
The limit exists because Cartesian products of caps give
$M(m+\ell)\ge M(m)M(\ell)$.  

\begin{theorem}[Ternary case, distance $r=3$]\label{thm:ternary-r3}
As $n\to\infty$,
\[
 T_3(n,3)
 \ge
 \left(\frac{\alpha}{27}-o(1)\right)
 \frac{\log\log n}{\log n}\,3^n.
\]
Along an infinite sequence of dimensions,
\[
 T_3(n,3)
 \ge
 \left(\frac{\alpha}{9}-o(1)\right)
 \frac{\log\log n}{\log n}\,3^n.
\]
\end{theorem}

For $q=3$, this improves the lower bound supplied by
Theorem~\ref{thm:checksum} to a density of order
$\log\log n/\log n$. The improvement comes from controlling
triangles directly, while allowing edges within each class.

In the opposite direction, Proposition~\ref{prop:distance-three-counting-upper}
gives an explicit constant-density upper bound for every fixed $q\ge3$.
In particular, we prove that
\[
 T_3(n,3)\le\left(\frac5{27}+O(n^{-1})\right)3^n.
\]
Whether this density tends to zero remains open.

For every prime power $q$, the extended BCH construction in
Proposition~\ref{prop:prime-power-bch} gives
\[
 T_q(n,r)=\Omega_{q,r}\!\left(
 \frac{q^n}{n^{\lfloor(q-1)r/q\rfloor}}\right)
 \qquad(r\ge2\text{ fixed}).
\]
At $q=3$ and $r=4$, it matches Proposition~\ref{prop:binary-transfer}
and yields the following exact order.

\begin{theorem}[Ternary case,  distance $r=4$]\label{thm:ternary-r4}
\[
       T_3(n,4)=\Theta\!\left(\frac{3^n}{n^2}\right).
\]
\end{theorem}

Theorems~\ref{thm:r1r2} and~\ref{thm:ternary-r4} suggest that the
binary-subcube upper bound may give the correct power of $n$ at every
fixed even distance. The general sparse hypergraph bound in
Theorem~\ref{thm:sparse-hypergraph} gives
$\Omega_{q,h}(q^n\sqrt{\log n}/n^{3h/2})$ for $q\ge3$.
The BCH construction improves the polynomial exponent in some cases,
but the even-distance upper bound is not matched in general.

\begin{conjecture}[Fixed even distances]\label{conj:even}
For every fixed $q\ge2$ and $h\ge1$,
\[
       T_q(n,2h)=\Theta_{q,h}\!\left(\frac{q^n}{n^h}\right).
\]
\end{conjecture}

The central open case is nevertheless the odd-distance ternary problem.
Even for $T_3(n,3)$, there is no vanishing-density upper bound matching
Theorem~\ref{thm:ternary-r3}.  For odd $r\ge5$, the local classification
introduces several exceptional-coordinate types simultaneously, and neither
binary restriction nor a single quadratic statistic controls them all.
Understanding how to combine higher-order algebraic labels with these local
types appears to be the main structural obstacle.

The paper is organized as follows. Section~\ref{sec: pre} establishes
the local classification, the averaging principles, the sparse
hypergraph lower bound, and the counting upper bound at distance three.
Section~\ref{sec:small-r} settles distances one and two and presents
the prime-power BCH construction, with ternary distance four as a
corollary. Section~\ref{sec:checksum} gives the distance-three
construction using relation-avoiding weights, and Section~\ref{sec:ternary-distance-three}
develops the ternary Ramsey construction. Finally, Section~\ref{sec:conc}
discusses the remaining questions and the converse connection to
multicolor Ramsey upper bounds.

\section{Preliminaries}\label{sec: pre}

\subsection{Local structure and averaging}
We first make precise the coordinatewise description used throughout the
paper. 

\begin{definition}
We call a coordinate \emph{active} for a triple of vectors if its three entries are
not all equal.
\end{definition}

\begin{proposition}[Local type classification]\label{prop:local-type}
Suppose that $x,y,z\in[q]^n$ satisfy
\[
 \dist(x,y)=\dist(x,z)=\dist(y,z)=r.
\]
At each active coordinate, either exactly one of $x,y,z$ is exceptional and
the other two agree, or all three symbols are distinct. If $a_x,a_y,a_z$
count the three exceptional-coordinate types and $d$ counts the
all-distinct coordinates, then
\[
 a_x=a_y=a_z=:a, \qquad 2a+d=r.
\]
In particular, $0\le a\le\lfloor r/2\rfloor$, $d=r-2a$, and the active
support has size $r+a$.
\end{proposition}

\begin{proof}
Every active coordinate contributes to exactly two of the three pairwise
distances if one symbol is exceptional, and to all three distances if the
symbols are pairwise distinct. Consequently
\[
 a_x+a_y+d=r,\qquad a_x+a_z+d=r,\qquad a_y+a_z+d=r.
\]
Subtracting these equations gives $a_x=a_y=a_z=a$, after which
$2a+d=r$. The support calculation is
$a_x+a_y+a_z+d=3a+d=r+a$.
\end{proof}

In the binary cube an all-distinct coordinate is impossible, so Proposition
\ref{prop:local-type} forces $d=0$ and $r=2a$. This recovers the familiar
parity obstruction to binary equilateral triangles. If $q\ge3$, every
admissible pair $(a,d)$ satisfying $r+a\le n$ is realizable.

We next give the two averaging arguments used to transfer bounds. Although
elementary, they are especially useful because they preserve density exactly.

\begin{lemma}[Coordinate-face averaging]\label{lem:face-average}
If $n\ge k\ge r$, then
\[
 \frac{T_q(n,r)}{q^n}\le \frac{T_q(k,r)}{q^k}.
\]
\end{lemma}

\begin{proof}
Let $A\subseteq[q]^n$ be triangle-free. Choose $k$ coordinates to remain
free and assign symbols independently and uniformly to the other $n-k$
coordinates. The resulting face is an isometric copy of $[q]^k$, and each
element of $A$ belongs to it with probability $q^{-(n-k)}$. Hence the
expected size of the intersection is $|A|q^{-(n-k)}$. Every such
intersection is triangle-free in $H_q(k,r)$, so
$|A|q^{-(n-k)}\le T_q(k,r)$.
\end{proof}

\begin{lemma}[Alphabet restriction]\label{lem:alphabet-restriction}
If $2\le p\le q$, then
\[
 \frac{T_q(n,r)}{q^n}\le \frac{T_p(n,r)}{p^n},
 \qquad\text{equivalently}\qquad
 T_q(n,r)\le\left(\frac qp\right)^nT_p(n,r).
\]
\end{lemma}

\begin{proof}
In each coordinate choose a uniformly random $p$-element subset of the
$q$-symbol alphabet, independently of the other coordinates. Their product
is an isometric copy of $[p]^n$. A fixed word lies in this copy with
probability $(p/q)^n$. Thus, for every triangle-free $A\subseteq[q]^n$, the
expected size of its intersection with the random copy is $(p/q)^n|A|$.
Every intersection has size at most $T_p(n,r)$, proving the result.
\end{proof}

Restriction-and-averaging arguments of this kind also occur in the
literature on perfect hashing and trifferent codes.  Bhandari and Khetan
average translates of a code against fixed symbol-weight layers, and note
the corresponding formulation for an arbitrary test set
\cite[Lemma~1.3 and Remark~1.4]{bhandari-khetan}.  Bishnoi, Kielak,
Kov\'acs, Nagy, Somlai, Vizer, and Zheng use the resulting translate-density
inequality together with coordinate-deletion recurrences in their
generalized trifference problem
\cite[Proposition~1.3 and Lemma~2.5]{bishnoi-et-al-trifference}.  These are
close variants of the same principle: a global extremal density is compared
with the extremal density on a structured family of restrictions.

Taking $p=2$ and applying \eqref{eq:BCL} proves
Proposition~\ref{prop:binary-transfer}.

\subsection{A general lower bound from sparse hypergraph coloring}

The local classification gives degree and codegree estimates for
$\mathcal H_{q,n,r}$. Applying the sparse hypergraph coloring theorem of
Cooper and Mubayi~\cite{cooper-mubayi} then improves the elementary
alteration bound by a factor of $\sqrt{\log n}$.

\begin{theorem}[Sparse hypergraph bound]\label{thm:sparse-hypergraph}
For fixed $q\ge3$ and fixed $r\ge1$,
\[
 T_q(n,r)=\Omega_{q,r}\!\left(
 \frac{q^n\sqrt{\log n}}{n^{(r+\lfloor r/2\rfloor)/2}}
 \right).
\]
For $q=2$ and fixed even $r$, the same argument gives
\[
 T_2(n,r)=\Omega_r\!\left(
 \frac{2^n\sqrt{\log n}}{n^{3r/4}}\right).
\]
\end{theorem}

\begin{proof}
Fix $x\in[q]^n$ and a local type $a$, and put $d=r-2a$.
To construct an ordered pair $(y,z)$ of this type, choose disjoint
coordinate sets of sizes $a,a,a,d$ for the three exceptional types
and the all-distinct type. There are $q-1$ choices at each
exceptional coordinate and $(q-1)(q-2)$ choices at each all-distinct
coordinate. Thus the number of ordered pairs is
\begin{equation}\label{eq:type-count}
 \frac{(n)_{r+a}}{(a!)^3d!}(q-1)^{r+a}(q-2)^d,
\end{equation}
where $(n)_m=n(n-1)\cdots(n-m+1)$.
Let $\mathcal H=\mathcal H_{q,n,r}$ and put $h=\lfloor r/2\rfloor$.
For fixed $q\ge3$ and $r$, summing \eqref{eq:type-count} over
$0\le a\le h$ and dividing by two gives
\[
 \Delta:=\Delta(\mathcal H)=\Theta_{q,r}(n^{r+h}).
\]

The maximum pair codegree $\Delta_2(\mathcal H)$ is the largest
number of hyperedges containing a fixed pair of distinct vertices.
A pair not at distance $r$ has codegree zero. Fix $x,y$ with
$\dist(x,y)=r$, and let $S=\{i:x_i\ne y_i\}$, so $|S|=r$.
For a common distance-$r$ neighbour $z$, let $a$ count the
coordinates outside $S$ at which $z$ differs from the common value
of $x,y$. Within $S$, let $b,c,d$ count, respectively, the
coordinates where $z=x$, where $z=y$, and where $z$ differs from
both. The distance equations are
\[
 b+c+d=r,\qquad a+c+d=r,\qquad a+b+d=r.
\]
Hence $b=c=a$ and $d=r-2a$, so the codegree of $\{x,y\}$ is
\begin{equation}\label{eq:pair-codegree}
 \sum_{a=0}^{\min\{h,n-r\}}
 \binom{n-r}{a}(q-1)^a
 \frac{r!}{(a!)^2(r-2a)!}(q-2)^{r-2a}
 =O_{q,r}(n^h).
\end{equation}

The $3$-uniform case of the theorem of Cooper and
Mubayi~\cite{cooper-mubayi} states that a hypergraph of maximum
degree $\Delta$ and maximum pair codegree at most $\Delta^{1/2}/f$
has chromatic number $O(\sqrt{\Delta/\log f})$ for sufficiently
large $f$. Here
$\Delta^{1/2}/n^h=\Theta_{q,r}(n^{(r-h)/2})$, so we may take
$f=n^{(r-h)/4}$ for all sufficiently large $n$. Since $r-h>0$,
\[
 \chi(\mathcal H)=O_{q,r}\!\left(
 \frac{n^{(r+h)/2}}{\sqrt{\log n}}\right).
\]
Each color class is independent in $\mathcal H$ and thus induces
a triangle-free subgraph of $H_q(n,r)$. A largest class has size
at least $q^n/\chi(\mathcal H)$, proving the first assertion.

When $q=2$ and $r$ is even, only the type $a=r/2$, $d=0$ survives.
Then $\Delta=\Theta_r(n^{3r/2})$, and a distance-$r$ pair has codegree
\[
 \binom{n-r}{r/2}\frac{r!}{((r/2)!)^2}=O_r(n^{r/2}).
\]
Taking $f=n^{r/8}$ proves the binary assertion, which remains
weaker than the sharp bound \eqref{eq:BCL}.
\end{proof}

For comparison, elementary alteration uses only the vertex-degree
estimate. With $h=\lfloor r/2\rfloor$, selecting vertices with
probability $p=c_{q,r}n^{-(r+h)/2}$ gives $pq^n$ expected selected
vertices and $O_{q,r}(p^3q^n n^{r+h})$ expected selected triangles.
For sufficiently small $c_{q,r}>0$, deleting one vertex from each
surviving triangle gives $\Omega_{q,r}(q^n/n^{(r+h)/2})$ vertices.
The codegree estimate supplies the additional factor $\sqrt{\log n}$.
For odd $r$, the polynomial exponent is $(3r-1)/4$; for $r=2h$ it is
$3h/2=3r/4$.

\subsection{An upper bound at distance three}
\label{subsec:distance-three-counting-upper}

\begin{proposition}\label{prop:distance-three-counting-upper}
For every fixed integer $q\ge3$,
\[
 T_q(n,3)\le
 \left(
 \frac{1+4(q-1)^3/(\max\{q,9\}-1)}{q(2q-3)^2}
 +O_q(n^{-1})
 \right)q^n.
\]
For the ternary alphabet, the following explicit bound holds:
\[
 T_3(n,3)\le
 \frac{5n-1}{27(n-1)}\,3^n
 =
 \left(\frac5{27}+\frac{4}{27(n-1)}\right)3^n
 \qquad(n\ge5).
\]
\end{proposition}

\begin{proof}
Let $A\subseteq[q]^n$ be nonempty and triangle-free in $H_q(n,3)$,
and let $e_j(A)$ count its unordered pairs at distance $j$.
For $n\ge4$, the graph contains both the $q$-clique
$\{(u,u,u,0,\ldots,0):u\in[q]\}$ and the $9$-clique
\[
 \{(u,v,u+v,u-v,0,\ldots,0):u,v\in\mathbb F_3\},
\]
embedded into three alphabet symbols. Averaging isometric copies of
the larger clique through any $x\in A$ gives
\begin{equation}\label{eq:distance-three-clique-count}
 2e_3(A)\le
 \frac{(q-1)^3\binom n3}{\max\{q,9\}-1}|A|.
\end{equation}
Indeed, each copy contains at most one further point of $A$, and the
isometries fixing $x$ act transitively on its distance-three neighbours.

For each coordinate line $\ell$, obtained by fixing all but one
coordinate, let $s_\ell$ count the points of $A$ on $\ell$, and let
$t_\ell$ count those at distance one from $\ell$.
There are $nq^{n-1}$ such lines, and
\[
 \sum_\ell s_\ell=n|A|,
 \qquad
 \sum_\ell t_\ell=n(n-1)(q-1)|A|.
\]
Counting pairs according to their distance gives
\begin{align*}
 \sum_\ell s_\ell^2
 &=n|A|+2e_1(A),\\
 \sum_\ell s_\ell t_\ell
 &=2(n-1)e_1(A)+4e_2(A),\\
 \sum_\ell t_\ell^2
 &=n(n-1)(q-1)|A|+2(n-1)(2q-3)e_1(A)\\
 &\qquad+4(n+q-4)e_2(A)+12e_3(A).
\end{align*}
In particular, two points at distance three have exactly six
coordinate lines at distance one from both: choose one of their
three differing coordinates as the line direction, then one of
two assignments on the remaining differing coordinates.

Expanding the square below cancels the $e_2(A)$ term and gives the 
coefficient $-3n^2/2+O_q(n)$ for $e_1(A)$.
Thus, for sufficiently large $n$, \eqref{eq:distance-three-clique-count}
yields
\begin{equation}\label{eq:distance-three-square-upper}
 \sum_\ell\left(t_\ell-\frac{n+q-4}{2}s_\ell\right)^2
 \le
 \left[
 \left(\frac14+\frac{(q-1)^3}{\max\{q,9\}-1}\right)n^3
 +O_q(n^2)
 \right]|A|.
\end{equation}
On the other hand, Cauchy--Schwarz and the first-moment identities give
\begin{align*}
 \sum_\ell\left(t_\ell-\frac{n+q-4}{2}s_\ell\right)^2
 &\ge
 \frac{n\bigl((n-1)(q-1)-(n+q-4)/2\bigr)^2}{q^{n-1}}|A|^2\\
 &=\left(\frac{q(2q-3)^2}{4}n^3+O_q(n^2)\right)
   \frac{|A|^2}{q^n}.
\end{align*}
Combining this with \eqref{eq:distance-three-square-upper} proves the general bound. For $q=3$, the same calculation gives an explicit bound.
Indeed, the exact moment identities yield
\begin{align*}
 \sum_\ell\left(t_\ell-\frac{n-1}{2}s_\ell\right)^2
 &=
 \frac{n(n-1)(n+7)}4\,|A|
 +\frac32(n-1)(5-n)e_1(A)
 +12e_3(A).
\end{align*}
For $n\ge5$, the coefficient of $e_1(A)$ is nonpositive, while
\eqref{eq:distance-three-clique-count} gives
\[
 12e_3(A)\le n(n-1)(n-2)|A|.
\]
Consequently, Cauchy--Schwarz and the first-moment identities give
\[
 \frac{27n(n-1)^2}{4}\frac{|A|^2}{3^n}
 \le
 \sum_\ell\left(t_\ell-\frac{n-1}{2}s_\ell\right)^2
 \le
 \frac{n(n-1)(5n-1)}4\,|A|.
\]
Dividing proves the explicit ternary bound.
\end{proof}

The clique input above is optimal: the classification of equidistant
distance-three codes by Bogdanova and Katona~\cite[Theorem~2.4]{bogdanova-katona}
implies
\[
 \omega(H_q(n,3))=\max\{q,9\}
 \qquad(q\ge3,\ n\ge4).
\]
Our proof uses only the explicit cliques, but their maximality shows
that the clique-averaging estimate
\eqref{eq:distance-three-clique-count} cannot be strengthened
by choosing a larger clique.

\section{Small distances and BCH constructions}\label{sec:small-r}

\subsection{The cases \texorpdfstring{$r=1$ and $r=2$}{r=1 and r=2}}

We start by proving Theorem~\ref{thm:r1r2}. At unit distance, the answer comes from
the geometry of coordinate lines.

\begin{proof}[Proof of Theorem~\ref{thm:r1r2} for $r=1$]
Fix a coordinate $i$ and all entries outside that coordinate. The resulting
coordinate line is a copy of $K_q$ in $H_q(n,1)$, and a triangle-free set
contains at most two of its points. The $q^{n-1}$ parallel coordinate lines
partition the cube, so $T_q(n,1)\le2q^{n-1}$.

For the reverse inequality, identify the alphabet with $\mathbb Z_q$
and take
\[
 A=\left\{x\in(\mathbb Z_q)^n:
            \sum_{i=1}^n x_i\in\{0,1\}\right\}.
\]
Every coordinate line meets $A$ in exactly two points. Hence $A$ is
triangle-free and $|A|=2q^{n-1}$.
\end{proof}

At distance two we use a checksum whose coefficients avoid all bounded
two-term relations. The construction works over any alphabet and produces
an independent set.

\begin{proof}[Proof of Theorem~\ref{thm:r1r2} for $r=2$]
Let $D_q=\{\pm1,\ldots,\pm(q-1)\}$. Choose a prime
$p=\Theta_q(n)$, with the implicit constant sufficiently large. On
$\F_p^\times$, join distinct $u,v$ whenever
\[
       c_1u+c_2v=0
       \qquad\text{for some }c_1,c_2\in D_q.
\]
This graph has maximum degree $O_q(1)$, because $u/v$ belongs to the fixed
set $\{-c_2/c_1:c_1,c_2\in D_q\}$. It therefore has an
independent set of order $\Omega_q(p)$; our choice of the implicit constant
allows us to select distinct elements $w_1,\ldots,w_n$ from it.

Label the alphabet by $0,1,\ldots,q-1$ and color $x\in[q]^n$ by
\[
       \chi(x)=\sum_{i=1}^n w_ix_i\pmod p.
\]
If equally colored words differed in exactly two coordinates $i,j$, then
$c_1w_i+c_2w_j=0$ for some $c_1,c_2\in D_q$, contrary to the
choice of the weights. Thus every color class is independent in $H_q(n,2)$.
There are $p=O_q(n)$ classes, so the largest has size
$\Omega_q(q^n/n)$. The matching upper bound follows from
Proposition~\ref{prop:binary-transfer} with $h=1$.
\end{proof}

\subsection{Extended BCH constructions over prime-power alphabets}

The following construction uses the standard parity-check equations
for extended Bose--Chaudhuri--Hocquenghem (BCH) codes,
see also~\cite{yekhanin-dumer} for the fixed-distance coding context.
We include the argument to record its consequences for triangle-free
induced subgraphs at every fixed distance over a prime-power alphabet.
A Vandermonde system excludes differences of small weight, while the
Frobenius identity reduces the number of moments that must be recorded.

\begin{proposition}[Prime-power BCH construction]\label{prop:prime-power-bch}
Let $q\ge2$ be a prime power and let $n\ge r\ge2$ be integers. Put
\[
 e=\left\lfloor\frac{(q-1)r}{q}\right\rfloor.
\]
There is a proper coloring of $H_q(n,r)$ with fewer than
$q^{e+1}n^e$ colors whose nonempty color classes have minimum
Hamming distance at least $r+1$. Consequently,
\begin{equation}\label{eq:prime-power-bch}
 T_q(n,r)>\frac{2}{q^{e+1}}\frac{q^n}{n^e}.
\end{equation}
\end{proposition}

\begin{proof}
Let $q^m$ be the smallest power of $q$ at least $n$, and set
$K=\F_{q^m}$. Thus $n\le |K|<qn$. Choose distinct elements
$a_1,\ldots,a_n\in K$, and put
\[
 J=\{j\in\{1,\ldots,r-1\}:q\nmid j\}.
\]
Then
\[
 |J|=r-1-\left\lfloor\frac{r-1}{q}\right\rfloor
     =\left\lfloor\frac{(q-1)r}{q}\right\rfloor=e.
\]
Identify the alphabet with $\F_q$. For $x\in\F_q^n$, write
\[
 S_0(x)=\sum_{i=1}^n x_i\in\F_q,
 \qquad S_j(x)=\sum_{i=1}^n a_i^jx_i\in K\quad(j\ge1),
\]
and color $x$ by
\[
 \Phi(x)=\bigl(S_0(x),(S_j(x))_{j\in J}\bigr)
          \in\F_q\times K^J.
\]
Suppose distinct vertices $x,y$ have the same color, and put $u=x-y$.
Since $u_i\in\F_q$, the Frobenius identity gives
$S_{qj}(u)=S_j(u)^q$ for every $j\ge1$.
Every positive integer at most $r-1$ is $q^t j$ for some $j\in J$,
so equality of colors implies
\[
 S_j(u)=0\qquad(0\le j\le r-1).
\]
If $w=\operatorname{wt}(u)\le r$ and
$\operatorname{supp}(u)=\{i_1,\ldots,i_w\}$, the first $w$ equations are
\[
 \sum_{\ell=1}^w a_{i_\ell}^{\,j}u_{i_\ell}=0
 \qquad(0\le j\le w-1),
\]
with all coefficients equal to one in the equation for $j=0$.
Their Vandermonde determinant is
$\prod_{s<t}(a_{i_t}-a_{i_s})\ne0$, forcing $u=0$, a contradiction.
Every nonempty fiber therefore has minimum distance at least $r+1$.

There are at most $q|K|^e<q^{e+1}n^e$ colors.
The map $\Phi$ is $\F_q$-linear, so its nonempty fibers have equal
size, at least $q^n/(q|K|^e)$. Its first coordinate is surjective,
so there are at least $q\ge2$ nonempty fibers. The union of two
distinct fibers induces a bipartite graph and has size at least
\[
 \frac{2q^n}{q|K|^e}>\frac{2}{q^{e+1}}\frac{q^n}{n^e},
\]
which proves \eqref{eq:prime-power-bch}.
\end{proof}

If $n$ is a power of $q$, taking $|K|=n$ improves the constant to
$T_q(n,r)\ge 2q^n/(qn^e)$.
At ternary distance four the construction matches the binary-subcube
upper bound, proving Theorem~\ref{thm:ternary-r4}.

\begin{corollary}[Ternary distance four]\label{cor:ternary-r4}
As $n\to\infty$,
\[
 T_3(n,4)=\Theta\!\left(\frac{3^n}{n^2}\right).
\]
\end{corollary}

\begin{proof}
Taking $q=3$ and $r=4$ in Proposition~\ref{prop:prime-power-bch} gives
$T_3(n,4)>2\cdot3^n/(27n^2)$ for $n\ge4$.
The matching upper bound follows from
Proposition~\ref{prop:binary-transfer}.
\end{proof}

\subsection{Excluding a fixed clique}

The averaging and coloring arguments also apply when a larger fixed
clique is forbidden.

\begin{corollary}\label{cor:fixed-clique-even-distances}
For every fixed $h\ge1$, $s\ge2$, and integer $q\ge2$,
\[
 \alpha_s(H_q(n,2h))=O_{h,s}\!\left(\frac{q^n}{n^h}\right).
\]
For every fixed $s\ge2$ and $q\ge2$,
\[
 \alpha_s(H_q(n,2))=\Theta_{q,s}\!\left(\frac{q^n}{n}\right),
 \qquad
 \alpha_s(H_3(n,4))=\Theta_s\!\left(\frac{3^n}{n^2}\right).
\]
More generally, for every fixed prime power $q$ and fixed $r,s\ge2$,
\[
 \alpha_s(H_q(n,r))=
 \Omega_{q,r,s}\!\left(\frac{q^n}{n^{\lfloor(q-1)r/q\rfloor}}\right).
\]
\end{corollary}

\begin{proof}
Let $A\subseteq[q]^n$ induce a $K_s$-free subgraph of $H_q(n,2h)$.
Intersecting $A$ with a uniformly random coordinatewise binary
subcube and taking expectations gives
\[
 \left(\frac2q\right)^n|A|\le\alpha_s(H_2(n,2h)).
\]
The upper bound follows from \eqref{eq:BCL}.
For the lower bounds, the union of at most $s-1$ classes of a
proper coloring is $K_s$-free. The distance-two checksum uses
$O_q(n)$ colors, and Proposition~\ref{prop:prime-power-bch} uses
$O_{q,r}(n^{\lfloor(q-1)r/q\rfloor})$ colors. Taking their largest
classes proves the lower bounds. For $q=3$ and $r=4$, the latter
exponent is two, matching the upper bound.
\end{proof}


\section{Distance \texorpdfstring{$r=3$, $q$}{r=3, q}-ary construction}
\label{sec:checksum}
We prove Theorem~\ref{thm:checksum} by constructing coordinate
weights with no bounded-coefficient relation involving three
distinct indices. The main ingredient is the limited-ratio-free set construction
of Erd\H{o}s, Frankl, and R\"odl \cite{erdos-frankl-rodl},
in the formulation of \cite[Theorem~2.5]{kovacs-nagy};
Alon and Shapira \cite{alon-shapira} developed a more general
framework for avoiding systems of linear equations. A translation
eliminates relations whose coefficients have nonzero sum, and a
suitable modulus preserves the resulting avoidance property.
The weighted coordinate sum then gives a proper coloring of
$H_q(n,3)$.

\begin{theorem}[Alon--Shapira, limited-ratio form]
\label{thm:limited-ratio}
For every fixed positive integer $h$, there is a constant $c_h>0$ such that
for every $N\ge2$ there is a set $B\subseteq\{1,\ldots,N\}$ satisfying
\[
 |B|\ge N\exp(-c_h\sqrt{\log N})
\]
and containing no three distinct elements $a,b,c$ for which
\[
       \lambda a+\mu b=(\lambda+\mu)c
       \qquad(1\le\lambda,\mu\le h).
\]
\end{theorem}

\begin{lemma}[Coefficient-avoiding weights]\label{lem:weights}
Let $D_q=\{\pm1,\ldots,\pm(q-1)\}$. For every fixed $q\ge2$ there is a constant $C_q>0$ with the following
property. For every $n\ge2$ one can find an integer
\[
       M\le n\exp(C_q\sqrt{\log n})
\]
and distinct residues $w_1,\ldots,w_n\in\mathbb Z/M\mathbb Z$ such that
\[
       c_1w_i+c_2w_j+c_3w_k\not\equiv0\pmod M
\]
whenever $i,j,k$ are distinct and
$c_1,c_2,c_3\in D_q$.
\end{lemma}

\begin{proof}
Apply Theorem~\ref{thm:limited-ratio} with $h=q-1$. Increasing the implicit
constant if necessary, choose
\[
       N\le n\exp(C'_q\sqrt{\log n})
\]
so that the theorem supplies a limited-ratio-free set
$B\subseteq\{1,\ldots,N\}$ of size at least $n$. Indeed, taking
$N=\lceil n\exp(3c_{q-1}\sqrt{\log n})\rceil$ suffices for all large $n$,
and the remaining values may be absorbed into the constant.

Set
\[
       H=3(q-1)N+1, \qquad W=H+B.
\]
We claim that distinct $u,v,w\in W$ satisfy no integer relation
\begin{equation}\label{eq:integer-relation}
       c_1u+c_2v+c_3w=0,
       \qquad c_1,c_2,c_3\in D_q.
\end{equation}
Write $u=H+a$, $v=H+b$, and $w=H+c$. If
$c_1+c_2+c_3=0$, then after permuting the variables and multiplying
by $-1$ if necessary, \eqref{eq:integer-relation} has the form
\[
       \lambda u+\mu v=(\lambda+\mu)w,
       \qquad 1\le\lambda,\mu\le q-1.
\]
The translation by $H$ cancels and produces a forbidden relation among
$a,b,c\in B$. If instead $s:=c_1+c_2+c_3\ne0$, then
\[
 |sH|\ge H
 \quad\text{but}\quad
 |c_1a+c_2b+c_3c|\le3(q-1)N<H,
\]
which again rules out \eqref{eq:integer-relation}.

Choose $n$ distinct elements of $W$, represented by positive integers, and
take $ M=3(q-1)(H+N)+1$.
Every relevant three-term sum has absolute value less than $M$; congruence
to zero modulo $M$ would therefore imply equality over the integers. The
chosen residues have the desired property, and
$M=O_q(N)\le n\exp(C_q\sqrt{\log n})$ after adjusting $C_q$.
\end{proof}

\begin{proof}[Proof of Theorem~\ref{thm:checksum}]
Take $M,w_1,\ldots,w_n$ from Lemma~\ref{lem:weights} and define
\[
 \chi:[q]^n\longrightarrow\mathbb Z_M,
 \qquad
 \chi(x)=\sum_{i=1}^n w_ix_i\pmod M.
\]
If $x$ and $y$ differ in exactly the three coordinates $i,j,k$ and have the
same color, then
\[
 (x_i-y_i)w_i+(x_j-y_j)w_j+(x_k-y_k)w_k\equiv0\pmod M.
\]
The three coefficients belong to $D_q$, contradicting
Lemma~\ref{lem:weights}. Hence $\chi$ is a proper coloring of
$H_q(n,3)$. At least one of its at most $M$ color classes has size at least
\[
 \frac{q^n}{M}
 \ge
 \frac{q^n}{n\exp(C_q\sqrt{\log n})},
\]
and this class is independent.
\end{proof}

The proof gives a proper coloring of $H_q(n,3)$, so every color
class is independent. In the next section we exploit the greater
freedom of requiring the color classes to induce triangle-free
subgraphs, using algebraic features specific to the ternary
alphabet.

\section{The ternary distance-three problem and Ramsey theory}
\label{sec:ternary-distance-three}

We now exploit features special to $\F_3$.  Proposition
\ref{prop:local-type} leaves two types at distance three.  For $a=0$ and
$d=3$, the triangle is an affine line
\[
       x,\quad x+u,\quad x-u,
       \qquad \operatorname{wt}(u)=3.
\]
The second type has $(a,d)=(1,1)$.  A linear statistic will eliminate the
mixed type, while a quadratic statistic will handle the affine lines.
Recall the cap-size function $M(m)$ and the constants $\kappa_3$
and $\alpha$ from \eqref{eq:cap-capacity}.
Cartesian products give $M(m+\ell)\ge M(m)M(\ell)$, so Fekete's
lemma yields $\kappa_3=\sup_m M(m)^{1/m}$ and
\begin{equation}\label{eq:cap-capacity-asymptotic}
 M(m)=3^{m(\alpha+o(1))}.
\end{equation}
In characteristic three, the cap condition is equivalent to 
\begin{equation}\label{eq:cap-zero-sum-characterization}
 a+b+c=0,\qquad a,b,c\in C,
 \quad\Longrightarrow\quad a=b=c.
\end{equation}
Indeed, if two of $a,b,c$ agree then all three agree; otherwise the three
points are distinct and form an affine line.

The edge-labeling parameter below is an additive analogue of a multicolor
Ramsey problem, tailored to the quadratic Hamming lift.

\begin{definition}[Zero-sum triangle Ramsey number] For an elementary abelian $3$-group $G$, let $Z(G)$ be the largest integer
$N$ for which the edges $ij\in E(K_N)$ of the complete graph $K_N$ can be labeled by elements
$\gamma_{ij}\in G$ so that
\begin{equation}\label{eq:zero-sum-label}
       \gamma_{ij}+\gamma_{ik}+\gamma_{jk}\ne0
\end{equation}
for every triple of distinct vertices $i,j,k$.
\end{definition} 

We next derive the required estimate for $Z(G)$.  This is the key quantitative ingredient in the distance-three lower bound.  The argument first packs the
group with disjoint cap palettes and then amplifies an ordinary multicolor
Ramsey coloring lexicographically.

\begin{lemma}[Palette amplification]\label{lem:zero-sum-palette}
Let $G$ be an elementary abelian $3$-group containing pairwise disjoint
caps
\[
 P_1,\ldots,P_L,\qquad |P_j|=s.
\]
If the edges of $K_t$ admit an $s$-coloring with no monochromatic triangles,
then
\[
 Z(G)\ge t^L.
\]
\end{lemma}

\begin{proof}
Fix such a coloring $\varphi:E(K_t)\to[s]$ and bijections
$\iota_j:[s]\to P_j$.  Use $[t]^L$ as the vertex set of a complete graph.
For distinct $u,v\in[t]^L$, let $j$ be the first coordinate at which they
differ and label the edge $uv$ by
\begin{equation}\label{eq:lexicographic-zero-sum-label}
 \gamma(u,v)=\iota_j\bigl(\varphi(u_jv_j)\bigr)\in P_j.
\end{equation}

Consider three distinct words $u,v,w$, and let $j$ be the first coordinate
at which they are not all equal.  If $u_j,v_j,w_j$ are pairwise distinct,
then all three edge labels belong to $P_j$.  They are not all equal, since
$\varphi$ has no monochromatic triangle.  By
\eqref{eq:cap-zero-sum-characterization}, their sum therefore cannot be
zero.

Suppose instead that exactly two entries agree, say $u_j=v_j\ne w_j$.
The edges $uw$ and $vw$ receive the same label $a\in P_j$, whereas the
edge $uv$ is decided at a later coordinate $j'>j$ and receives a label
$b\in P_{j'}$.  The palettes are disjoint, so $a\ne b$; hence, in
characteristic three,
\[
 a+a+b=b-a\ne0.
\]
Thus no triangle has label sum zero, proving that $Z(G)\ge |[t]^L|=t^L$.
\end{proof}

It remains to produce sufficiently many disjoint caps. Partitioning $\mathbb F_3^m$ into caps is equivalent to
properly coloring the affine Steiner triple system
$\mathrm{AG}(m,3)$.
The following elementary translate-covering argument is
a special case of the general relationship between
fractional and integral covers.
We include the proof to keep the subsequent construction
self-contained.

\begin{lemma}[Partition into caps]\label{lem:zero-sum-cap-partition}
Let $C\subseteq \F_3^m$ be a cap, put $Q=3^m$, and write
$M=|C|$.  Then $\F_3^m$ can be partitioned into at most
\begin{equation}\label{eq:number-of-cap-parts}
 J\le \left\lceil \frac{Q}{M}(\log Q+1)\right\rceil
\end{equation}
caps. 
\end{lemma}

\begin{proof}
Choose $J$ independent uniformly random translates of $C$, with $J$ equal
to the right-hand side of \eqref{eq:number-of-cap-parts}.  A fixed element
of $\F_3^m$ is missed by all the translates with probability
\[
 \left(1-\frac{M}{Q}\right)^J
 \le \exp(-JM/Q).
\]
Consequently the expected number of uncovered elements is at most
\[
 Q\exp(-JM/Q)\le e^{-1}<1.
\]
Some collection of $J$ translates therefore covers $\F_3^m$.  Order these
translates and assign every element to the first translate containing it.
The resulting nonempty parts are pairwise disjoint subsets of translates of
$C$, hence are caps, and they partition $\F_3^m$.
\end{proof}

\begin{lemma}[Equal cap palettes]\label{lem:equal-cap-palettes}
Fix $0<\beta<\alpha$.  For $Q=3^m$ and
\[
 s=\lfloor Q^\beta\rfloor,
\]
there are
\begin{equation}\label{eq:number-of-equal-cap-palettes}
 L=(1-o(1))\frac{Q}{s}
\end{equation}
pairwise disjoint caps of size $s$, as $m\to\infty$.
\end{lemma}

\begin{proof}
Choose $\beta'$ with $\beta<\beta'<\alpha$.  From the definition of
$\kappa_3$ and $\alpha$ we have
\[
 M(m)=Q^{\alpha+o(1)},
\]
so, for all sufficiently large $m$, $\F_3^m$ contains a cap of size at
least $Q^{\beta'}$.  Lemma~\ref{lem:zero-sum-cap-partition} partitions
$\F_3^m$ into
\[
 J\le Q^{1-\beta'+o(1)}
\]
caps.  Split every part into sets of size $s$, discarding the remainder of
each part.  Fewer than $Js$ elements are discarded, and
\[
 Js\le Q^{1-\beta'+\beta+o(1)}=o(Q).
\]
The remaining $Q-o(Q)$ elements therefore form the asserted family of
$L=(1-o(1))Q/s$ disjoint $s$-element caps.
\end{proof}

Let $R_s(3)$ be the least $N$ such that every $s$-coloring of
$E(K_N)$ contains a monochromatic triangle. We use the following
superexponential lower bound from
\cite[Chapter~9]{openai-ten-advances}, announced recently by OpenAI.

\begin{theorem}[Superexponential multicolor Ramsey bound]
\label{thm:multicolor-triangle-lower-bound}
There is an absolute constant $c>0$ such that, for every integer $s\ge2$,
\begin{equation}\label{eq:multicolor-triangle-lower-bound}
 R_s(3)\ge
 \left(\frac{c\,s^{1/3}}{\log s}\right)^s.
\end{equation}
Consequently, as $s\to\infty$, there is an $s$-coloring of some $K_t$
with no monochromatic triangle and
\begin{equation}\label{eq:ramsey-order-consequence}
 \log t\ge \left(\frac13-o(1)\right)s\log s.
\end{equation}
\end{theorem}

This result enables us to prove the zero-sum Ramsey estimate that drives the ternary
distance-three construction.

\begin{theorem}[Zero-sum Ramsey amplification]
\label{thm:zero-sum-amplification}
Let $Q=3^m$.  Then
\begin{equation}\label{eq:zero-sum-amplification}
 Z(\F_3^m)\ge Q^{(\alpha/3-o(1))Q}
 \qquad (m\to\infty).
\end{equation}
\end{theorem}

\begin{proof}
Fix $0<\beta<\alpha$ and put $s=\lfloor Q^\beta\rfloor$.  By
Lemma~\ref{lem:equal-cap-palettes}, the group $\F_3^m$ contains
\[
 L=(1-o(1))\frac{Q}{s}
\]
pairwise disjoint caps of size $s$.  By
Theorem~\ref{thm:multicolor-triangle-lower-bound}, there is an $s$-coloring
of $K_t$ with no monochromatic triangle and
\[
 \log t\ge \left(\frac13-o(1)\right)s\log s.
\]
Lemma~\ref{lem:zero-sum-palette} gives $Z(\F_3^m)\ge t^L$, and hence
\begin{align*}
 \log Z(\F_3^m)
 &\ge L\log t\\
 &\ge (1-o(1))\frac{Q}{s}
       \left(\frac13-o(1)\right)s\log s\\
 &=\left(\frac{\beta}{3}-o(1)\right)Q\log Q.
\end{align*}
Thus $Z(\F_3^m)\ge Q^{(\beta/3-o(1))Q}$.  Since $\beta<\alpha$ is arbitrary,
\eqref{eq:zero-sum-amplification} follows.
\end{proof}
\begin{corollary}\label{cor:zero-sum-growth}
As $m\to\infty$,
$  \log Z(\mathbb F_3^m)=\Theta(m3^m).$

More precisely, writing $Q=3^m$, we have
$  Q^{(\alpha/3-o(1))Q}
  \le Z(\mathbb F_3^m)
  \le Q^{(1+o(1))Q}.$
\end{corollary}

\begin{proof}
The lower bound is Theorem~\ref{thm:zero-sum-amplification}.
A labeling by $\mathbb F_3^m$ can be regarded as an
edge-coloring with $Q$ colors. Every monochromatic
triangle has label sum zero, so
$Z(\mathbb F_3^m)<R_Q(3)$.
The classical factorial upper bound
$R_Q(3)=O(Q!)$ gives the required upper estimate.
\end{proof}
\Needspace{8\baselineskip}
\begin{theorem}[Zero-sum Ramsey input]\label{thm:zero-sum-input}
Let $\kappa_3$ and $\alpha$ be defined by \eqref{eq:cap-capacity}.  For every
sufficiently large $n$ there is an elementary abelian $3$-group $G$ with
$n\le Z(G)$ and
\[
 |G|\le\left(\frac9\alpha+o(1)\right)
              \frac{\log n}{\log\log n}.
\]
Along an infinite sequence of values of $n$, one may replace $9/\alpha$ by
$3/\alpha$.
\end{theorem}

\begin{proof}
Fix $0<\eta<\alpha/3$.  Theorem~\ref{thm:zero-sum-amplification} implies that
for every sufficiently large power $Q=3^m$,
\begin{equation}\label{eq:fixed-eta-zero-sum-bound}
 Z(\F_3^m)\ge Q^{\eta Q}.
\end{equation}
Given a sufficiently large $n$, choose the least power of three $Q$ for
which
\[
 n\le Q^{\eta Q}.
\]
Then \eqref{eq:fixed-eta-zero-sum-bound} supplies the required labeling in
$G=\F_3^m$.  By the minimality of $Q$, the preceding power $Q/3$ satisfies
\[
 n>\left(\frac Q3\right)^{\eta Q/3},
\]
and therefore
\begin{equation}\label{eq:zero-sum-inversion}
 \log n>\frac{\eta Q}{3}\log\frac Q3.
\end{equation}
Together with $\log n\le\eta Q\log Q$, this shows that
\[
 Q=\Theta\!\left(\frac{\log n}{\log\log n}\right)
 \quad\text{and}\quad
 \log Q=(1+o(1))\log\log n.
\]
Substitution in \eqref{eq:zero-sum-inversion} yields
\[
 |G|=Q\le
 \left(\frac3\eta+o(1)\right)
 \frac{\log n}{\log\log n}.
\]
Letting $\eta$ tend to $\alpha/3$ proves the first assertion.

For the subsequence assertion, let $Q=3^m$ tend to infinity.  From
\eqref{eq:zero-sum-amplification} we may choose integers $n_m\le
Z(\F_3^m)$ satisfying
\[
 \log n_m=\left(\frac\alpha3-o(1)\right)Q\log Q.
\]
It follows that
\[
 Q=\left(\frac3\alpha+o(1)\right)
 \frac{\log n_m}{\log\log n_m},
\]
which proves the sharper estimate along the infinite sequence
$(n_m)_{m\ge1}$.
\end{proof}

The available numerical value of $\kappa_3$ is governed by a succession of explicit
constructions.  Edel's generalized product caps yielded the longstanding
bound $\kappa_3\ge 2.2174$ \cite{edel}.  Tyrrell obtained the first
improvement in nearly two decades, reaching $2.218$
\cite{tyrrell}; the program-search construction of Romera-Paredes et al.\
subsequently reached $2.2202$ \cite{funsearch}.  The construction of Zhai et al.~\cite{x-evolve} gives the explicit
lower bound
\[
                  \kappa_3\ge 2.2203
                  \qquad\text{and hence}\qquad
                  \alpha\ge 0.72604.
\]
Consequently the two constants in
Theorem~\ref{thm:ternary-r3} may presently be taken to be at least
\[
        \frac{\alpha}{27}\ge 0.02689,
        \qquad
        \frac{\alpha}{9}\ge 0.08067,
\]
respectively.

The constants in Theorem~\ref{thm:ternary-r3} now follow from an explicit
coloring. We include the reduction because it exposes the role of the two
local types.

\begin{proof}[Proof of Theorem~\ref{thm:ternary-r3}]
Choose $G$ and a labeling $\gamma_{ij}$ satisfying
\eqref{eq:zero-sum-label}.  For $x\in\F_3^n$ define
\[
 \sigma(x)=\sum_{i=1}^n x_i\in\F_3,
 \qquad
 Q(x)=\sum_{1\le i<j\le n}x_ix_j\gamma_{ij}\in G,
\]
and color $x$ by $(\sigma(x),Q(x))$.  This gives at most $3|G|$ color
classes.

Suppose first that $x,y,z$ form a triangle of type $(a,d)=(1,1)$.
Adding the same field element to all three entries in any coordinate
shifts their $\sigma$-values equally, so it preserves whether these
values coincide. We may therefore normalize every inactive coordinate
to zero and the repeated symbol in each exceptional coordinate to zero.
The three exceptional coordinates then contribute
$(e_x,0,0)$, $(0,e_y,0)$, and $(0,0,e_z)$, respectively, with
$e_x,e_y,e_z\ne0$. At the unique all-distinct coordinate, the entries
$(t_x,t_y,t_z)$ form a permutation of $\F_3$. Hence
\[
 \sigma(x)=e_x+t_x,\qquad
 \sigma(y)=e_y+t_y,\qquad
 \sigma(z)=e_z+t_z.
\]
If these three values were equal to $\tau$, then
$e_x=\tau-t_x$, $e_y=\tau-t_y$, and $e_z=\tau-t_z$.
Since one of $t_x,t_y,t_z$ equals $\tau$, one of
$e_x,e_y,e_z$ would be zero, a contradiction.
Thus a mixed triangle is not monochromatic.

It remains to consider a type-$(0,3)$ triangle.  Write it as
$x,x+u,x-u$, with $u$ supported on three coordinates $i,j,k$.  If the
three vertices have the same $\sigma$-value, then $\sum_i u_i=0$.  Since the
three nonzero entries of $u$ lie in $\{1,-1\}$, they must all be equal.
The quadratic second difference is then
\[
 Q(x+u)+Q(x-u)-2Q(x)
   =2\sum_{a<b}u_au_b\gamma_{ab}
   =-\bigl(\gamma_{ij}+\gamma_{ik}+\gamma_{jk}\bigr),
\]
where the last equality is in characteristic three. This is nonzero by
\eqref{eq:zero-sum-label}.  Hence the three $Q$-values cannot all agree, and
the coloring is triangle-free.

Its largest class has size at least $3^n/(3|G|)$.  The first estimate in
Theorem~\ref{thm:zero-sum-input} gives
\[
 \frac{3^n}{3|G|}
 \ge
 \left(\frac{\alpha}{27}-o(1)\right)
 \frac{\log\log n}{\log n}\,3^n,
\]
and the improved estimate along the distinguished subsequence gives the
constant $\alpha/9$.
\end{proof}

The construction uses exactly the algebra demanded by the local geometry:
the linear coordinate excludes the mixed type and the quadratic coordinate
excludes affine lines. In contrast, the proper coloring constructed in
Section~\ref{sec:checksum} separates
every distance-three pair and is therefore uniform in $q$ but much less
dense when $q=3$.

\section{Concluding remarks}\label{sec:conc}

At odd distances, binary restriction supplies no nontrivial upper
bound, since every binary odd-distance graph is bipartite. Over
$\F_3$, the type-$a=0$ triangles are affine lines $x,x+u,x-u$ with
$\operatorname{wt}(u)=r$. The essential difficulty is to control
these together with all types $a\ge1$. At distance three there is
only one additional type, and a single linear statistic removes it.
For larger odd distances, a replacement may require higher-order
labels that encode several exceptional-coordinate patterns at once.

We conclude with three open questions. First, can the lower bounds of
Theorem~\ref{thm:sparse-hypergraph} and
Proposition~\ref{prop:prime-power-bch} be improved to match the
binary-subcube upper bound at every fixed even distance, as proposed
in Conjecture~\ref{conj:even}? Second, can the constant-density upper
bound of Proposition~\ref{prop:distance-three-counting-upper} be
replaced by a vanishing-density bound? In particular, is
$T_3(n,3)=o(3^n)$, improving the current estimate
\[
 T_3(n,3)\le\left(\frac5{27}+O(n^{-1})\right)3^n?
\]
Third, can the triangle-free condition be used to improve on the
BCH and sparse hypergraph bounds at larger odd distances?
The ternary distance-three construction suggests that controlling
triangles directly can be substantially more effective than
excluding every edge.

There is also a quantitative connection between the upper bounds of $T_3(n,3)$ and the multicolor
triangle Ramsey numbers $R_k(3)$ defined in
Section~\ref{sec:ternary-distance-three}.
Suppose that
$C\subseteq\mathbb F_3^m$ is a cap of size at least $k$. Thus 
$ c_1+c_2+c_3=0,\qquad c_1,c_2,c_3\in C,$
implies $c_1=c_2=c_3$. Replacing the colors of a $k$-edge-coloring
without monochromatic triangles by distinct elements of $C$
therefore produces an edge-labeling without zero-sum triangles.
Consequently,
\[
 R_k(3)-1\le Z(\mathbb F_3^m).
\]

The quadratic construction gives the complementary implication
\[
 n\le Z(\mathbb F_3^m)
 \quad\Longrightarrow\quad
 T_3(n,3)\ge 3^{n-m-1}.
\]
Indeed, the construction partitions the hyperplane
$\{x\in\mathbb F_3^n:\sum_i x_i=0\}$ into at most $3^m$
triangle-free fibers, one of which has size at least
$3^{n-m-1}$. Combining these observations yields
\begin{equation}\label{eq:hamming-to-ramsey}
 T_3(n,3)<3^{n-m-1}
 \quad\Longrightarrow\quad
 R_k(3)\le n,
\end{equation}
whenever $\mathbb F_3^m$ contains a cap of size at least $k$.
Thus sufficiently strong upper bounds for the distance-three
problem would imply upper bounds for classical multicolor
triangle Ramsey numbers.

The rate of any decay in $T_3(n,3)/3^n$ is therefore relevant
to Ramsey theory, beyond the qualitative question of whether
this density tends to zero. The implications above transfer
upper bounds; a positive-density lower bound for $T_3(n,3)$
alone would not yield a converse Ramsey estimate.

The problem of forbidding larger cliques has already been
considered in this setting. Balogh, Chen, and
Li~\cite[Question~1.2 and Theorem~1.4]{balogh-chen-li}
determined the order of $\alpha_s(H_2(n,r))$ for fixed $s$
and fixed even $r$. Mukkamala and
Ravi~\cite[Section~4, Problem~3]{mukkamala-ravi}
explicitly posed the corresponding problem over prime-power
alphabets. The constructions above settle the case
$q=r=3$ and $s=4$, as $
 \alpha_4(H_3(n,3))=3^{n-1}$ for $n\ge4$. Indeed, the proof of Theorem~\ref{thm:ternary-r3} shows that each
fiber of the linear statistic $\sum_i x_i$ contains no triangle
of mixed type. Every triangle in such a fiber is therefore an
affine line. A $K_4$ is impossible, since its triangles through
a fixed edge would have to lie on the same affine line, which
contains only three points. Each fiber thus gives a $K_4$-free
set of size $3^{n-1}$.
Conversely, the cosets of the two-dimensional subspace displayed
in the proof of Proposition~\ref{prop:distance-three-counting-upper}
partition $\F_3^n$ into $9$-cliques. A $K_4$-free set meets each
coset in at most three points, proving the matching upper bound.
In particular, $\alpha_s(H_3(n,3))\ge3^{n-1}$ for every $s\ge4$,
so the question of vanishing density at ternary distance three
is specific to the triangle-free problem.

\noindent \textbf{AI declaration.} AI-assisted tools were used in editing and preparing the manuscript: during the workshop, the authors used  ChatGPT 5.6 Sol, which assisted them in obtaining the main results of the manuscript by building on their ideas. Later, they used  ChatGPT 6 Astra for proofreading.
The authors reviewed the resulting text and take full responsibility for the mathematical statements, proofs, and references.

 \noindent \textbf{Acknowledgement.} The problem studied in this paper was posed at the Excellence
Workshop, held in Szigliget, Hungary, on 4--9 August 2026.
The authors thank the organizers for providing the opportunity
for these discussions.

\end{document}